\documentclass[12pt]{amsart}
\usepackage{}

\usepackage{amsmath}
\usepackage{amsfonts}
\usepackage{amssymb}
\usepackage[all]{xy}           

\usepackage{bbding}
\usepackage{txfonts}
\usepackage{amscd}

\usepackage[shortlabels]{enumitem}
\usepackage{ifpdf}
\ifpdf
  \usepackage[colorlinks,final,backref=page,hyperindex]{hyperref}
\else
  \usepackage[colorlinks,final,backref=page,hyperindex,hypertex]{hyperref}
\fi
\usepackage{tikz}
\usepackage[active]{srcltx}

\makeatletter

\newtheorem{thm}{Theorem}[section]
\newtheorem{lem}[thm]{Lemma}
\newtheorem{cor}[thm]{Corollary}
\newtheorem{pro}[thm]{Proposition}
\newtheorem{ex}[thm]{Example}
\newtheorem{rmk}[thm]{Remark}
\newtheorem{defi}[thm]{Definition}

\newcommand {\emptycomment}[1]{}

\newcommand{\be }{\begin{equation}}
\newcommand{\ee }{\end{equation}}

\newcommand{\huaL}{\mathcal{L}}

\newcommand{\huaO}{{\mathcal{O}}}

\newcommand{\frks}{\mathfrak s}

\newcommand{\Id}{\rm{Id}}

\newcommand{\br}[1]{   [ \cdot,    \cdot  ]   }

\newcommand{\gl}{\mathfrak {gl}}

\newcommand{\ad}{\mathrm{ad}}

\begin{document}

\title[On Hom-pre-Poisson algebras]{On Hom-pre-Poisson algebras
}

\author{Shanshan Liu}
\address{College of Mathematics and Systems Science, Shandong University of Science and Technology, Qingdao 266590, Shandong, China}
\email{shanshanmath@163.com}

\author{Abdenacer Makhlouf}
\address{University of Haute Alsace, IRIMAS- D\'epartement  de Math\'ematiques, Universit\'e de Haute Alsace, France}
\email{abdenacer.makhlouf@uha.fr}

\author{Lina Song}
\address{Department of Mathematics, Jilin University, Changchun 130012, Jilin, China}
\email{songln@jlu.edu.cn}


\begin{abstract}
In this paper, first we discuss  Hom-pre-Poisson algebras and their relationships with  Hom-Poisson algebra. Then we introduce  the notion of a Hom-pre-Gerstenhaber algebra and show that a Hom-pre-Gerstenhaber algebra gives rise to a Hom-Gerstenhaber algebra. Moreover, we consider  Hom-dendriform formal deformations of Hom-zinbiel algebras and show that Hom-pre-Poisson algebras are the corresponding semi-classical limits. Furthermore, we consider  Hom-$\huaO$-operators on Hom-Poisson algebras and study their  relationships with Hom-pre-Poisson algebras. Finally, we define the notion of  dual-Hom-pre-Poisson algebra and show that a Hom-average-operator on a Hom-Poisson algebra naturally gives rise to a dual-Hom-pre-Poisson algebra.
\end{abstract}


\maketitle
\vspace{-5mm}
\tableofcontents

\allowdisplaybreaks


\section{Introduction}
The notion of a Hom-Lie algebra was introduced by Hartwig, Larsson
and Silvestrov in \cite{HLS} as part of a study of deformations of
the Witt and the Virasoro algebras. In a Hom-Lie algebra, the Jacobi
identity is twisted by a linear map (homomorphism), and it is called the  Hom-Jacobi identity. Different types of Hom-algebras were introduced and widely studied.  Moreover, the bialgebra theory for Hom-algebras was deeply studied in \cite{Cai-Sheng,MS3,sheng1,Yao1}.

Pre-Lie algebras (also called  left-symmetric algebras,
quasi-associative algebras, Vinberg algebras and so on) are a
class of nonassociative algebras that  appeared in
many fields in mathematics and mathematical physics.
See  the survey \cite{Pre-lie algebra in geometry} and the references therein for more details. The notion of a left-symmetric bialgebra was introduced in \cite{Bai}. The author also introduced the notion of   $\frks$-matrix   to produce left-symmetric bialgebras.
The notion of a Hom-pre-Lie algebra was introduced in \cite{MS2}  and play important roles in the study of Hom-Lie bialgebras and Hom-Lie 2-algebras \cite{sheng1,SC}. Recently, Hom-pre-Lie algebras were studied from several aspects. Cohomologies of Hom-pre-Lie algebras were studied in \cite{LST}; The geometrization of Hom-pre-Lie algebras was studied in \cite{Qing}; Universal $\alpha$-central extensions of Hom-pre-Lie algebras were studied in \cite{sunbing}; Hom-pre-Lie bialgebras were studied in \cite{LSS,QH}

The notion of a pre-Poisson algebra was introduced by Aguiar in \cite{AGUIAR} , which combines a zinbiel algebra and a pre-Lie algebra such that some compatibility conditions are satisfied. zinbiel algebras were introduced by Loday in \cite{Loady} in his study of the
algebraic structure behind the cup product on the cohomology groups of a Leibniz algebra. A pre-Poisson algebra gives
rise to a Poisson algebra naturally through the sub-adjacent commutative associative algebra of the zinbiel algebra and the sub-adjacent Lie algebra of the pre-Lie algebra. Conversely, an $\huaO$-operator action on a Poisson algebra gives rise to a pre-Poisson algebra. In \cite{LJF}, the author introduced the notion of a pre-$F$-manifold algebra, which also contains a zinbiel algebra and a pre-Lie algebra, such that some compatibility conditions are satisfied.  Aguiar's  pre-Poisson algebras are pre-$F$-manifold algebras as special cases. The notion of a Hom-Poisson algebra was introduced by Makhlouf and Silvestrov in \cite{MS1}, which combines a commutative Hom-associative algebra and a Hom-Lie algebra such that a Hom-Leibniz identity is satisfied. In \cite{Yao3}, the author introduced the notion of a non-commutative Hom-Poisson algebra, which combines a Hom-associative algebra (not necessarily commutative)  and a Hom-Lie algebra such that a Hom-Leibniz identity is satisfied. A Hom-type version of pre-$F$-manifold algebras studied in \cite{LJF} was considered recently  in \cite{BCMM}.

The purpose of this paper is to give a systematic and a specific study of  Hom-pre-Poisson algebras. We deal with the relationships between   Hom-pre-Poisson algebras and  Hom-pre-Poisson algebras, as well as  Hom-pre-Gerstenhaber algebras and Hom-pre-Gerstenhaber algebras. We introduce the notion of Hom-dendriform formal deformations of Hom-zinbiel algebras and show that Hom-pre-Poisson algebras are the corresponding semi-classical
limits. We give the notion of Hom-$\huaO$-operators on Hom-Poisson algebras and study its relation with Hom-pre-Poisson algebras. Finally, we define a notion of  dual-Hom-pre-Poisson algebra and  Hom-average-operator on a Hom-Poisson algebra. We show that a Hom-average-operator on a Hom-Poisson algebra lead to  a dual-Hom-pre-Poisson algebra.

The paper is organized as follows. In Section \ref{sec:pre}, we recall relevant definitions and representation about Hom-algebras. In Section \ref{sec:HPP}, we consider  the notion of a Hom-pre-Poisson algebra and study  its relationship with Hom-Poisson algebras. Moreover, we introduce the notion of Hom-pre-Gerstenhaber algebras, and study its relationship with Hom-Gerstenhaber algebras. In Section \ref{sec:def}, we consider Hom-dendriform formal deformations of Hom-zinbiel algebras and show that Hom-pre-Poisson algebras are the corresponding semi-classical limits. In Section \ref{sec:O}, we discuss the notion of Hom-$\huaO$-operators on Hom-Poisson algebras and corresponding  Hom-pre-Poisson algebra. Conversely, we show that a  Hom-pre-Poisson algebra naturally gives a Hom-$\huaO$-operator on the sub-adjacent Hom-Poisson algebra. In Section \ref{sec:dual}, we deal with dual-Hom-pre-Poisson algebras and  Hom-average-operators on a Hom-Poisson algebras. We show that a Hom-average-operator on a Hom-Poisson algebra gives rise to  a dual-Hom-pre-Poisson algebra.

\vspace{2mm}
\noindent
{\bf Acknowledgements. }  We give warmest thanks to Yunhe Sheng   for helpful comments that improve the paper.   This work is supported by
 National Natural Science Foundation of China (Grant No. 12001226).

\section{Preliminaries}\label{sec:pre}

\begin{defi}{\rm(\cite{MS2})}
A {\bf Hom-associative algebra} is a triple $(A,\cdot,\alpha)$ consisting of a vector space $A$, a bilinear map $\cdot:A\otimes A\longrightarrow A$ and an algebra morphism $\alpha:A\longrightarrow A$, satisfying:
  \begin{equation}
\alpha(x)\cdot(y\cdot z)=(x\cdot y)\cdot \alpha(z),\quad
\forall~x,y,z\in A.
\end{equation}
A Hom-associative algebra for which $x\cdot y=y\cdot x$ is called a {\bf commutative Hom-associative algebra}. A commutative Hom-associative algebra $(A,\cdot,\alpha)$ is said to be {\bf regular} if $\alpha$ is invertible.
 \end{defi}

\begin{defi}
 A {\bf representation} of a commutative Hom-associative algebra $(A,\cdot,\alpha)$ on
 a vector space $V$ with respect to $\beta\in\gl(V)$ is a linear map
  $\mu:A\longrightarrow \gl(V)$, such that for all
  $x,y\in A$, the following equalities are satisfied:
\begin{eqnarray}
\label{hom-comm-rep-1}\mu(\alpha(x))\circ \beta&=&\beta\circ \mu(x),\\
\label{hom-comm-rep-2}\mu(x\cdot y)\circ \beta&=&\mu(\alpha(x))\circ\mu(y).
\end{eqnarray}
  \end{defi}
We denote a representation by $(V,\beta,\mu)$. Moreover, define a linear map $L:A\longrightarrow \gl(A)$ by $L_x (y)=x\cdot y$. Then $(A,L,\alpha)$ is a representation of the commutative Hom-associative algebra $(A,\cdot,\alpha)$, which is called the {\bf regular representation}.

Let $(V,\beta,\mu)$ be a representation of a commutative Hom-associative algebra $(A,\cdot,\alpha)$. In the sequel, we always assume that $\beta$ is invertible. For all $x\in A,u\in V,\xi\in V^*$, define $\mu^*:A\longrightarrow\gl(V^*)$ as usual by
$$\langle\mu^*(x)(\xi),u\rangle=-\langle\xi,\mu(x)(u)\rangle.$$

Then define $\mu^\star:A\longrightarrow\gl(V^*)$ by
\begin{equation}\label{eq:1.3}
 \mu^\star(x)(\xi):=\mu^*(\alpha(x))\big{(}(\beta^{-2})^*(\xi)\big{)}.
\end{equation}

\begin{thm}\label{dual-rep-1}
Let $(V,\beta,\mu)$ be a representation of a commutative Hom-associative algebra $(A,\cdot,\alpha)$, where $\beta$ is invertible. Then $(V^*,(\beta^{-1})^*,-\mu^\star)$ is a representation of $(A,\cdot,\alpha)$, which is called the {\bf dual representation}.
\end{thm}
\begin{proof}
For all $x\in A,\xi\in V^*$ and $u\in V$, by \eqref{hom-comm-rep-1}, we have
\begin{eqnarray*}
-\mu^\star(\alpha(x))\big((\beta^{-1})^*(\xi)\big)&=&-\mu^*(\alpha^{2}(x))\big((\beta^{-3})^*(\xi)\big)\\
&=&(\beta^{-1})^*\big(-\mu^*(\alpha(x))\big)\big((\beta^{-2})^*(\xi)\big)\\
&=&(\beta^{-1})^*\big(-\mu^\star(x)(\xi)\big),
\end{eqnarray*}
which implies that
\begin{equation}\label{hom-dual-rep-1}
-\mu^\star{(}\alpha(x){)}\circ(\beta^{-1})^*=(\beta^{-1})^*\circ(-\mu^\star(x)).
\end{equation}

For all $x,y\in A,\xi\in V^*$ and $u\in V$, by \eqref{hom-comm-rep-2}, we have
\begin{eqnarray*}
\Big\langle-\mu^\star(x\cdot y)\big((\beta^{-1})^*(\xi)\big),u\Big\rangle
&=&\Big\langle-\mu^*\big(\alpha(x)\cdot \alpha(y)\big)\big((\beta^{-3})^*(\xi)\big),u\Big\rangle\\
&=&\Big\langle-\mu^*\big(\alpha(y)\cdot \alpha(x)\big)\big((\beta^{-3})^*(\xi)\big),u\Big\rangle\\
&=&\Big\langle(\beta^{-3})^*(\xi),\mu\big(\alpha(y)\cdot \alpha(x)\big)(u)\Big\rangle\\
&=&\Big\langle(\beta^{-3})^*(\xi),\mu(\alpha^{2}(y))\big(\mu(\alpha (x))(\beta^{-1}(u))\big)\Big\rangle\\
&=&\Big\langle(\beta^{-4})^*(\xi),\mu(\alpha^{3}(y))\big(\mu(\alpha^{2}(x))(u)\big)\Big\rangle\\
&=&\Big\langle\mu^*(\alpha^{2}(x))\big(\mu^*(\alpha^{3}(y))((\beta^{-4})^*(\xi))\big),u\Big\rangle\\
&=&\Big\langle\mu^*(\alpha^{2}(x))\big(\mu^\star(\alpha^{2}(y))((\beta^{-2})^*(\xi))\big),u\Big\rangle\\
&=&\Big\langle\mu^*(\alpha^{2}(x))\big((\beta^{-2})^*(\mu^\star(y)(\xi))\big),u\Big\rangle\\
&=&\Big\langle\mu^\star(\alpha (x))\big{(}\mu^\star(y)(\xi)\big{)},u\Big\rangle,
\end{eqnarray*}
which implies that
\begin{equation}\label{hom-dual-rep-2}
-\mu^\star(x\cdot y)\circ(\beta^{-1})^*=\mu^\star(\alpha(x))\circ\mu^\star(y).
\end{equation}
By \eqref{hom-dual-rep-1} and \eqref{hom-dual-rep-2}, we deduce that $(V^*,(\beta^{-1})^*,-\mu^\star)$ is a representation of $(A,\cdot,\alpha)$.
\end{proof}

\begin{defi}{\rm(\cite{HLS})}
A {\bf Hom-Lie algebra} is a triple $(A,[\cdot,\cdot],\alpha)$ consisting of a linear space $A$, a skew-symmetric bilinear map $[\cdot,\cdot]:\wedge^2A\longrightarrow A$ and an algebra morphism $\alpha:A\longrightarrow A $, satisfying:
\begin{equation}
[\alpha(x),[y,z]]+[\alpha(y),[z,x]]+[\alpha(z),[x,y]]=0,\quad
\forall~x,y,z\in A.
\end{equation}
 A Hom-Lie algebra $(A,[\cdot,\cdot],\alpha)$ is said to be {\bf regular} if $\alpha$ is invertible.
  \end{defi}
\begin{defi}{\rm(\cite{BM,sheng3})}\label{defi:hom-lie representation}
 A {\bf representation} of a Hom-Lie algebra $(A,[\cdot,\cdot],\alpha)$ on
 a vector space $V$ with respect to $\beta\in\gl(V)$ is a linear map
  $\rho:A\longrightarrow \gl(V)$, such that for all
  $x,y\in A$, the following equalities are satisfied:
\begin{eqnarray}
\label{hom-lie-rep-1}\rho(\alpha(x))\circ \beta&=&\beta\circ \rho(x),\\
\label{hom-lie-rep-2}\rho([x,y])\circ \beta&=&\rho(\alpha(x))\circ\rho(y)-\rho(\alpha(y))\circ\rho(x).
\end{eqnarray}
  \end{defi}
We denote a representation by $(V,\beta,\rho)$.  Moreover, define a linear map $\ad:A\longrightarrow \gl(A)$ by $\ad_x (y)=[x,y]$. Then $(A,\alpha,\ad)$ is a representation of the Hom-Lie algebra $(A,[\cdot,\cdot],\alpha)$, which is called the {\bf adjoint representation}

Let $(V,\beta,\rho)$ be a representation of a Hom-Lie algebra $(A,[\cdot,\cdot],\alpha)$. In the sequel, we always assume that $\beta$ is invertible. For all $x\in A,u\in V,\xi\in V^*$, define $\rho^*:A\longrightarrow\gl(V^*)$ as usual by
$$\langle\rho^*(x)(\xi),u\rangle=-\langle\xi,\rho(x)(u)\rangle$$
Then define $\rho^\star:A\longrightarrow\gl(V^*)$ by
\begin{equation}\label{eq:1.4}
 \rho^\star(x)(\xi):=\rho^*(\alpha(x))\big{(}(\beta^{-2})^*(\xi)\big{)}.
\end{equation}
\begin{thm}{\rm(\cite{Cai-Sheng})}\label{dual-rep-2}
Let $(V,\beta,\rho)$ be a representation of a Hom-Lie algebra $(A,[\cdot,\cdot],\alpha)$, where $\beta$ is invertible. Then $(V^*,(\beta^{-1})^*,\rho^\star)$ is a representation of $(A,[\cdot,\cdot],\alpha)$, which is called the {\bf dual representation}.
\end{thm}

\begin{defi}{\rm(\cite{MS2})}
A {\bf Hom-pre-Lie algebra} is a triple $(A,\ast,\alpha)$ consisting of a vector space $A$, a bilinear map $\ast:A\otimes A\longrightarrow A$ and an algebra morphism $\alpha:A\longrightarrow A $, satisfying:
\begin{eqnarray}
(x\ast y)\ast \alpha(z)-\alpha(x)\ast (y\ast z)=(y\ast x)\ast \alpha(z)-\alpha(y)\ast (x\ast z),\quad
\forall~x,y,z\in A.
\end{eqnarray}
\end{defi}
Let $(A,\ast,\alpha)$ be a Hom-pre-Lie algebra. The commutator $[x,y]=x\ast y-y\ast x$ gives a Hom-Lie algebra $(A,[\cdot,\cdot],\alpha)$, which is denoted by $A^C$ and called the sub-adjacent Hom-Lie algebra of $(A,\ast,\alpha)$. Moreover, define a linear map $\huaL:A\longrightarrow \gl(A)$ by $\huaL_x(y)=x\ast y$. Then $(A,\huaL,\alpha)$ is a representation of the sub-adjacent Hom-Lie algebra $A^C$

\begin{defi}{\rm(\cite{AEM})}
A {\bf Hom-zinbiel algebra} is a triple $(A,\diamond,\alpha)$ consisting of a vector space $A$, a bilinear map $\diamond:A\otimes A\longrightarrow A$ and an algebra morphism $\alpha:A\longrightarrow A $, satisfying:
\begin{equation}
\alpha(x)\diamond(y\diamond z)=(x\diamond y)\diamond\alpha(z)+(y\diamond x)\diamond\alpha(z),\quad
\forall~x,y,z\in A.
\end{equation}
\end{defi}
Defining $\cdot:A\otimes A\longrightarrow A$ by $x\cdot y=x\diamond y+y\diamond x$, we obtain a commutative Hom-associative algebra $(A,\cdot,\alpha)$. Moreover, define a linear map $\mathfrak{L}:A\longrightarrow \gl(A)$ by $\mathfrak{L}_x(y)=x \diamond y$. Then $(A,\mathfrak{L},\alpha)$ is a representation of the commutative Hom-associative algebra $(A,\cdot,\alpha)$.

\begin{defi}{\rm(\cite{MakDend})}
A {\bf Hom-dendriform algebra} is a quadruple $(A,\prec,\succ,\alpha)$ consisting of a vector space $A$, bilinear maps $\prec,\succ:A\otimes A\longrightarrow A$ and an algebra morphism $\alpha:A\longrightarrow A $, such that for all $x,y,z\in A$, satisfying:
\begin{eqnarray}
(x\prec y)\prec\alpha(z)&=&\alpha(x)\prec(y\prec z+y\succ z),\\
(x\succ y)\prec \alpha(z)&=&\alpha(x)\succ(y\prec z),\\
\alpha(x)\succ(y\succ z)&=&(x\prec y+x\succ y)\succ\alpha(z).
\end{eqnarray}
\end{defi}
A Hom-dendriform algebra for which $x\succ y=y\prec x=x \diamond y$ is exactly a Hom-zinbiel algebra.

Defining $\cdot:A\otimes A\longrightarrow A$ by $x\cdot y=x\succ y+x\prec y$, we obtain a Hom-associative algebra structure on $A$. Defining $\ast:A\otimes A\longrightarrow A$ by $x\ast y=x\succ y-y\prec x$, we obtain a Hom-pre algebra structure on $A$.

The above algebras can be summarized by the following diagram:
 $$
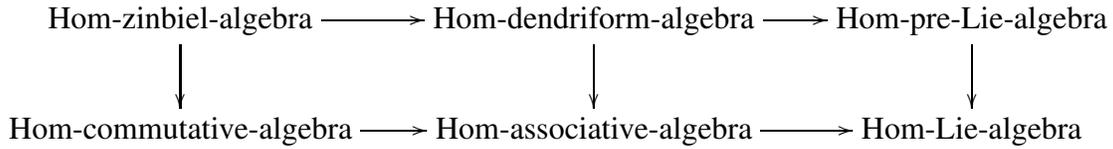
\xymatrix{
 \text{Hom-zinbiel-algebra}\ar[d]\ar[r]& \text{Hom-dendriform-algebra} \ar[d] \ar[r] & \text{Hom-pre-Lie-algebra} \ar[d] \\
  \text{Hom-commutative-algebra} \ar[r] &\text{Hom-associative-algebra} \ar[r] & \text{Hom-Lie-algebra}  }
$$

\section{Hom-pre-Poisson algebras and Hom-pre-Gerstenhaber algebras}\label{sec:HPP}

In this section, first we study representations of a Hom-Poisson algebra. Then we introduce the notion of a Hom-pre-Poisson algebra and give the relation between Hom-Poisson algebras and Hom-pre-Poisson algebras. Finally we introduce the notion of Hom-pre-Gerstenhaber algebras, and study its relation with Hom-Gerstenhaber algebras.
\subsection{Hom-pre-Poisson algebras}
\begin{defi}{\rm(\cite{MS1})}
A {\bf Hom-Poisson algebra} is a quadruple $(A,\cdot,[\cdot,\cdot],\alpha)$, where $(A,\cdot,\alpha)$ is a commutative Hom-associative algebra and $(A,[\cdot,\cdot],\alpha)$ is a Hom-Lie algebra, satisfying:
\begin{equation}\label{hom-poisson}
[\alpha(x),y\cdot z]=[x,y]\cdot \alpha(z)+\alpha(y)\cdot [x,z],\quad
\forall~x,y,z\in A.
\end{equation}
\end{defi}
\begin{ex}
Let $(A,\cdot_A,[\cdot,\cdot]_A,\alpha_A)$ and $(B,\cdot_B,[\cdot,\cdot]_B,\alpha_B)$ be two Hom-Poisson algebras. Then $(A\oplus B,\cdot_{A\oplus B},[\cdot,\cdot]_{A\oplus B},\alpha_A\oplus \alpha_B)$ is a Hom-Poisson algebra, where the bracket $[\cdot,\cdot]_{A\oplus B}$, the product $\cdot_{A\oplus B}$ and $\alpha_A\oplus \alpha_B$ are given by
\begin{eqnarray*}
[x+u,y+v]_{A\oplus B}&=&[x,y]_A+[u,v]_B,\\
(x+u)\cdot_{A\oplus B}(y+v)&=&x\cdot_A y+u\cdot_B v,\\
(\alpha_A\oplus \alpha_B)(x+u)&=&\alpha_A(x)+\alpha_B(u),
\end{eqnarray*}
for all $x,y\in A, u,v\in B.$
\end{ex}

\begin{defi}\label{defi:hom-pre-poisson representation}
 A {\bf representation} of a Hom-Poisson algebra $(A,\cdot,[\cdot,\cdot],\alpha)$ on a vector space $V$ with respect to $\beta\in\gl(V)$ consists of a pair $(\rho,\mu)$, where $(V,\beta,\rho)$ is a representation of the Hom-Lie algebra $(A,[\cdot,\cdot],\alpha)$ and $(V,\beta,\mu)$ is a representation of the commutative Hom-associative algebra $(A,\cdot,\alpha)$, such that for all $x,y\in A$, satisfying:
\begin{eqnarray}
\label{rep-1}\rho(x\cdot y)\circ \beta&=&\mu(\alpha(y))\circ \rho(x)+\mu(\alpha(x))\circ \rho(y),\\
\label{rep-2}\rho(\alpha(x))\circ \mu(y)&=&\mu(\alpha(y))\circ\rho(x)+\mu([x,y])\circ \beta.
\end{eqnarray}
\end{defi}
We denote a representation of a Hom-Poisson algebra $(A,\cdot,[\cdot,\cdot],\alpha)$ by $(V,\beta,\rho,\mu)$. Moreover, $(A,\alpha,\ad,L)$ is a representation of a Hom-Poisson algebra $(A,\cdot,[\cdot,\cdot],\alpha)$, which is called the {\bf regular representation}.
\begin{pro}
Let $(A,\cdot,[\cdot,\cdot],\alpha)$ be a Hom-Poisson algebra and $(V,\beta,\rho,\mu)$ its representation. Then $(A\oplus V,\cdot_{\mu},[\cdot,\cdot]_{\rho},\alpha\oplus \beta)$ is a Hom-Poisson algebra, where $(A\oplus V,\cdot_{\mu},\alpha\oplus \beta)$ is the semi-direct commutative Hom-associative algebra,
\begin{equation}
\nonumber(x+u)\cdot_{\mu}(y+v)=x\cdot y+\mu(x)v+\mu(y)u,\quad
\forall~x,y\in A, u,v\in V,
\end{equation}
and $(A\oplus V,[\cdot,\cdot]_{\rho},\alpha\oplus \beta)$ is the semi-direct Hom-Lie algebra,
\begin{equation}
\nonumber[x+u,y+v]_{\rho}=[x,y]+\rho(x)v-\rho(y)u,\quad
\forall~x,y\in A, u,v\in V.
\end{equation}
\end{pro}
\begin{proof}
For all $x,y,z\in A$, by \eqref{rep-2}, we have
\begin{eqnarray*}
[\alpha(x),y\cdot_{\mu} u]_{\rho}-[x,y]\cdot_{\mu}\beta(u)-\alpha(y)\cdot_{\mu}[x,u]_{\rho}
&=&[\alpha(x),\mu(y)u]_{\rho}-\mu([x,y])\beta(u)-\alpha(y)\cdot_{\mu}\rho(x)u\\
&=&\rho(\alpha(x))\mu(y)u-\mu([x,y])\beta(u)-\mu(\alpha(y))\rho(x)u\\
&=&0.
\end{eqnarray*}
Similarly, we have
\begin{equation}
\nonumber[\beta(u),x\cdot y]_{\rho}=[u,x]_{\rho}\cdot_{\mu}\alpha(y)+\alpha(x)\cdot_{\mu}[u,y]_{\rho}.
\end{equation}
This finishes the proof.
\end{proof}

\begin{pro}\label{dual-representation}
Let $(V,\beta,\rho,\mu)$ be a representation of a Hom-Poisson algebra $(A,\cdot,[\cdot,\cdot],\alpha)$, where $\beta$ is invertible. Then $(V^*,(\beta^{-1})^*,\rho^\star,-\mu^\star)$ is a representation of the Hom-Poisson algebra $(A,\cdot,[\cdot,\cdot],\alpha)$, which is called the {\bf dual representation}.
\end{pro}

\begin{proof}
By Theorem \ref{dual-rep-1} and Theorem \ref{dual-rep-2}, we know that $(V^*,(\beta^{-1})^*,-\mu^\star)$ and $(V^*,(\beta^{-1})^*,\rho^\star)$ are the dual representation of the commutative Hom-associative algebra $(A,\cdot,\alpha)$ and the Hom-Lie algebra $(A,[\cdot,\cdot],\alpha)$ respectively. For all $x,y\in A, \xi\in A^*, u\in V$, by \eqref{rep-2}, we have
\begin{eqnarray*}
&&\Big\langle-\mu^\star[x, y]\big((\beta^{-1})^*(\xi)\big),u\Big\rangle\\
&=&\Big\langle-\mu^*\big[\alpha(x),\alpha(y)\big]\big((\beta^{-3})^*(\xi)\big),u\Big\rangle\\
&=&\Big\langle(\beta^{-3})^*(\xi),\mu\big[\alpha(x),\alpha(y)\big](u)\Big\rangle\\
&=&\Big\langle(\beta^{-3})^*(\xi),\rho(\alpha^{2}(x))\big(\mu(\alpha (y))(\beta^{-1}(u))\big)-\mu(\alpha^{2}(y))\big(\rho(\alpha(x))(\beta^{-1}(u))\big)\rangle\\
&=&\Big\langle(\beta^{-4})^*(\xi),\rho(\alpha^{3}(x))\big(\mu(\alpha^{2}(y))(u)\big)-\mu(\alpha^{3}(y))\big(\rho(\alpha^{2}(x))(u)\big)\rangle\\
&=&\Big\langle\mu^*(\alpha^{2}(y))\big(\rho^*(\alpha^{3}(x))((\beta^{-4})^*(\xi))\big)-\rho^*(\alpha^{2}(x))\big(\mu^*(\alpha^{3}(y))((\beta^{-4})^*(\xi))\big),u\rangle\\
&=&\Big\langle\mu^*(\alpha^{2}(y))\big(\rho^\star(\alpha^{2}(x))((\beta^{-2})^*(\xi))\big)-\rho^*(\alpha^{2}(x))\big(\mu^\star(\alpha^{2}(y))((\beta^{-2})^*(\xi))\big),u\rangle\\
&=&\Big\langle\mu^*(\alpha^{2}(y))\big((\beta^{-2})^*(\rho^\star(x)(\xi))\big)-\rho^*(\alpha^{2}(x))\big((\beta^{-2})^*(\mu^\star(y)(\xi))\big),u\rangle\\
&=&\Big\langle\mu^\star(\alpha (y))\big{(}\rho^\star(x)(\xi)\big{)}-\rho^\star(\alpha (x))\big{(}\mu^\star(y)(\xi)\big{)},u\Big\rangle,
\end{eqnarray*}
which implies that
\begin{equation}\label{rep-3}
-\rho^\star(\alpha(x))\circ\mu^\star(y)=-\mu^\star(\alpha(y))\circ\rho^\star(x)-\mu^\star[x,y]\circ (\beta^{-1})^*.
\end{equation}
Similarly, by \eqref{rep-1} and \eqref{rep-3}, we have
\begin{equation}\label{rep-4}
\rho^\star(x\cdot y)\circ (\beta^{-1})^*=-\mu^\star(\alpha(y))\circ\rho^\star(x)-\mu^\star(\alpha(x))\circ\rho^\star(y).
\end{equation}
Thus, by \eqref{rep-3} and \eqref{rep-4}, $(V^*,(\beta^{-1})^*,\rho^\star,-\mu^\star)$ is a representation of the Hom-Poisson algebra $(A,\cdot,[\cdot,\cdot],\alpha)$.
\end{proof}

\begin{defi}
A {\bf Hom-pre-Poisson algebra} is a quadruple $(A,\diamond,\ast,\alpha)$, where $(A,\diamond,\alpha)$ is a Hom-zinbiel algebra and $(A,\ast,\alpha)$ is a Hom-pre-Lie algebra, such that for all $x,y,z\in A$, the following conditions hold:
\begin{eqnarray}
\label{hom-pre-poisson-1}(x\ast y-y\ast x)\diamond \alpha(z)&=&\alpha(x)\ast(y\diamond z)-\alpha(y)\diamond(x\ast z),\\
\label{hom-pre-poisson-2}(x\diamond y+y\diamond x)\ast \alpha(z)&=&\alpha(x)\diamond(y\ast z)+\alpha(y)\diamond (x\ast z).
\end{eqnarray}
A Hom-pre-Poisson algebra $(A,\diamond,\ast,\alpha)$ is said to be {\bf regular} if $\alpha$ is invertible.
\end{defi}
\begin{rmk}
When $\alpha=\Id$, we recover the pre-Poisson algebra given in \cite{AGUIAR}.
\end{rmk}
\begin{defi}
Let $(A,\diamond,\ast,\alpha)$ and $(A',\diamond',\ast',\alpha')$ be two Hom-pre-Poisson algebras. A morphism between $A$ and $A'$ is a linear map $f:A\longrightarrow A'$, such that for all $x,y\in A$, the following conditions hold:
\begin{eqnarray}
f(x\diamond y)&=&f(x)\diamond' f(y),\\
f(x\ast y)&=&f(x)\ast' f(y),\\
f\circ \alpha&=&\alpha'\circ f.
\end{eqnarray}
\end{defi}
\begin{thm}\label{relation}
Let $(A,\diamond,\ast,\alpha)$ be a Hom-pre-Poisson algebra. Define
$$x\cdot y=x\diamond y+y\diamond x,\quad [x,y]=x\ast y-y\ast x, \quad \forall x,y \in A.$$
Then $(A,\cdot,[\cdot,\cdot],\alpha)$ is a Hom-Poisson algebra, which is denoted by $A^C$ and called the sub-adjacent Hom-Poisson algebra of $(A,\diamond,\ast,\alpha)$, $(A,\diamond,\ast,\alpha)$ is called the compatible Hom-pre-Poisson algebra of $A^C$. Moreover, $(A,\alpha,\huaL,\mathfrak{L})$ is a representation of $A^C$, where $\huaL_x(y)=x\ast y$ and $\mathfrak{L}_x(y)=x\diamond y$.
\end{thm}
\begin{proof}
Obviously, $(A,\cdot,\alpha)$ is a commutative Hom-associative algebra and $(A,[\cdot,\cdot],\alpha)$ is a Hom-Lie algebra. For all $x,y,z\in A$, by the definition of a  Hom-pre-Poisson algebra, we have
\begin{eqnarray}
\nonumber &&[\alpha(x),y\cdot z]-[x,y]\cdot \alpha(z)-\alpha(y)\cdot [x,z]\\
\nonumber&=&[\alpha(x),y\diamond z+z\diamond y]-(x\ast y-y\ast x)\cdot \alpha(z)-\alpha(y)\cdot (x \ast z-z\ast x)\\
\nonumber &=&\alpha(x)\ast(y\diamond z)+\alpha(x)\ast(z\diamond y)-(y\diamond z)\ast \alpha(x)-(z\diamond y)\ast \alpha(x)\\
\nonumber&&-(x\ast y)\diamond\alpha(z)+(y\ast x)\diamond\alpha(z)-\alpha(z)\diamond(x\ast y)+\alpha(z)\diamond(y\ast x)\\
\nonumber&&-\alpha(y)\diamond(x\ast z)+\alpha(y)\diamond(z\ast x)-(x\ast z)\diamond\alpha(y)+(z\ast x)\diamond\alpha(y)\\
\nonumber &=&0.
\end{eqnarray}
which implies that
$$[\alpha(x),y\cdot z]=\alpha(y)\cdot [x,z]+[x,y]\cdot \alpha(z).$$
Thus, $(A,\cdot,[\cdot,\cdot],\alpha)$ is a Hom-Poisson algebra.

The triple $(A,\mathfrak{L},\alpha)$ defines  a representation of the commutative Hom-associative algebra $(A,\cdot,\alpha)$ and $(A,\huaL,\alpha)$ is a representation of the Hom-Lie algebra $(A,[\cdot,\cdot],\alpha)$. For all $x,y,z\in A$, by the definition of a Hom-pre-Poisson algebra, we have
\begin{eqnarray}
\nonumber &&\huaL_{x\cdot y}\alpha(z)-\mathfrak{L}_{\alpha(y)}\huaL_x(z)-\mathfrak{L}_{\alpha(x)}\huaL_y(z)\\
\nonumber&=&(x\cdot y)\ast \alpha(z)-\alpha(y)\diamond (x\ast z)-\alpha(x)\diamond(y\ast z)\\
\nonumber &=&(x\diamond y)\ast \alpha(z)+(y\diamond x)\ast \alpha(z)-\alpha(y)\diamond(x\ast z)-\alpha(x)\diamond(y\ast z)\\
\nonumber &=&0,
\end{eqnarray}
which implies that
$$\huaL_{x\cdot y} \circ\alpha=\mathfrak{L}_{\alpha(y)} \circ \huaL_x+\mathfrak{L}_{\alpha(x)} \circ \huaL_y.$$
Similarly, we have
$$\huaL_{\alpha(x)}\circ \mathfrak{L}_y=\mathfrak{L}_{\alpha(y)}\circ \huaL_x+\mathfrak{L}_{[x,y]}\circ \alpha.$$
Thus, $(A,\alpha,\huaL,\mathfrak{L})$ is a representation of $A^C$. This finishes the proof.
\end{proof}
At the end of this section, we establish the relation between pre-Poisson algebras and Hom-pre-Poisson algebras.

\begin{pro}\label{prepoisson-homprepoisson}
Let $(A,\diamond,\ast)$ be a pre-Poisson algebra and $\alpha:A\longrightarrow A$ an algebra endomorphism. Define $\diamond_{\alpha}=\alpha\circ \diamond$ and $\ast_{\alpha}=\alpha\circ\ast$. Then $(A,\diamond_{\alpha},\ast_{\alpha},\alpha)$ is a Hom-pre-Poisson algebra. Moreover, suppose that $(A',\diamond',\ast')$ is another pre-Poisson algebra and $\alpha':A'\longrightarrow A'$ is an algebra endomorphism. If $f:A\longrightarrow A'$ is a pre-Poisson algebra morphism that satisfies $f\circ\alpha=\alpha'\circ f$, then $f:(A,\diamond_{\alpha},\ast_{\alpha},\alpha)\longrightarrow (A',\diamond'_{\alpha'},\ast'_{\alpha'},\alpha')$ is a morphism of Hom-pre-Poisson algebras.
\end{pro}
\begin{proof}
For all $x,y,z\in A$, by the definition of a zinbiel algebra, we have
\begin{eqnarray*}
&&\alpha(x)\diamond_{\alpha}(y\diamond_{\alpha}z)-(x\diamond_{\alpha}y)\diamond_{\alpha}\alpha(z)-(y\diamond_{\alpha}x)\diamond_{\alpha}\alpha(z)\\
&=&\alpha(x)\diamond_{\alpha}(\alpha(y)\diamond \alpha(z))-(\alpha(x)\diamond \alpha(y))\diamond_{\alpha}\alpha(z)-(\alpha(y)\diamond \alpha(x))\diamond_{\alpha}\alpha(z)\\
&=&\alpha(\alpha(x)\diamond(\alpha(y)\diamond\alpha(z))-(\alpha(x)\diamond\alpha(y))\diamond\alpha(z)-(\alpha(y)\diamond\alpha(x))\diamond\alpha(z))\\
&=&0,
\end{eqnarray*}
which implies that $(A,\diamond_{\alpha},\alpha)$ is a Hom-zinbiel algebra. Similarly, we obtain that $(A,\ast_{\alpha},\alpha)$ is a Hom-pre-Lie algebra. By computation, we obtain that \eqref{hom-pre-poisson-1} and \eqref{hom-pre-poisson-2} hold. Therefore, $(A,\diamond_{\alpha},\ast_{\alpha},\alpha)$ is a Hom-pre-Poisson algebra.

For all $x,y\in A$, we have
\begin{equation}
\nonumber f(x\diamond_{\alpha}y)=f(\alpha(x)\diamond \alpha(y))=f(\alpha(x))\diamond' f(\alpha(y))=\alpha'(f(x))\diamond' \alpha'(f(y))=f(x)\diamond'_{\alpha'}f(y).
\end{equation}
Similarly, we have $f(x\ast_{\alpha}y)=f(x)\ast'_{\alpha'}f(y)$. Therefore, $f$ is a morphism from $(A,\diamond_{\alpha},\ast_{\alpha},\alpha)$ to $(A',\diamond'_{\alpha'},\ast'_{\alpha'},\alpha')$.
\end{proof}
\begin{cor}
Let $(A,\diamond,\ast,\alpha)$ be a regular Hom-pre-Poisson algebra. Define $\diamond_{\alpha^{-1}}=\alpha^{-1}\circ \diamond$ and $\ast_{\alpha^{-1}}=\alpha^{-1}\circ\ast$. Then $(A,\diamond_{\alpha^{-1}},\ast_{\alpha^{-1}})$ is a pre-Poisson algebra.
\begin{proof}
The proof is similar to Proposition \ref{prepoisson-homprepoisson}, we omit the details.
\end{proof}
\end{cor}
\subsection{Hom-pre-Gerstenhaber algebras}

The notion of Hom-Gerstenhaber algebra was introduced in \cite{LGT} in the study of Hom-Lie algebroids.
\begin{defi}{\rm(\cite{LGT})}
A Hom-Gerstenhaber algebra is a quadruple $(A,\cdot,[\cdot,\cdot],\alpha)$, where $A$ is a graded vector space, $(A,\cdot,\alpha)$ is a graded commutative Hom-associative algebra of degree $0$, that is the following identities hold
\begin{equation}
\nonumber x\cdot y=(-1)^{|x||y|}y\cdot x,\quad\alpha(x)\cdot(y\cdot z)=(x \cdot y)\cdot\alpha(z), \quad\forall x,y,z \in A,
\end{equation}
$(A,[\cdot,\cdot],\alpha)$ is a graded Hom-Lie algebra of degree $-1$, that is the following identities hold
\begin{equation}
\nonumber[x,y]=-(-1)^{(|x|-1)(|y|-1)}[y,x], \quad [\alpha(x),[y,z]]=[[x,y],\alpha(z)]+(-1)^{(|x|-1)(|y|-1)}[\alpha(y),[x,z]],
\end{equation}
and $\alpha:A\longrightarrow A$ is an algebra morphism of degree $0$,
such that the following identity holds:
\begin{equation}
\nonumber [\alpha(x),y\cdot z]=[x,y]\cdot\alpha(z)+(-1)^{(|x|-1)|y|}\alpha(y)\cdot[x,z],
\end{equation}
\end{defi}

\begin{defi}
A Hom-pre-Gerstenhaber algebra is a quadruple $(A,\diamond,\ast,\alpha)$, where $A$ is a graded vector space, $(A,\diamond,\alpha)$ is a graded Hom-zinbiel algebra of degree $0$, that is the following identities hold
\begin{equation}\label{Hom-zinbiel}
 \alpha(x)\diamond(y\diamond z)=(-1)^{|x||y|}(y\diamond x)\diamond\alpha(z)+(x\diamond y)\diamond\alpha(z), \quad\forall x,y \in A,
\end{equation}
$(A,\ast,\alpha)$ is a graded Hom-pre-Lie algebra of degree $-1$,  that is the following identities hold

\begin{equation}
\nonumber(x\ast y)\ast\alpha(z)-\alpha(x)\ast(y\ast z)=(-1)^{(|x|-1)(|y|-1)}((y\ast x)\ast\alpha(z)-\alpha(y)\ast(x\ast z)),\quad\forall x,y,z\in A,
\end{equation}
and $\alpha:A\longrightarrow A$ is an algebra morphism of degree $0$, such that the following identities hold:
\begin{eqnarray}
  \label{eq:pre-gerstenhaber-1}(x\ast y-(-1)^{(|x|-1)(|y|-1)}y\ast x)\diamond\alpha(z)&=&\alpha(x)\ast(y\diamond z)-(-1)^{(|x|-1)|y|}\alpha(y)\diamond(x\ast z),\\
   \label{eq:pre-gerstenhaber-2}(x\diamond y+(-1)^{|x||y|}y\diamond x)\ast \alpha(z)&=&\alpha(x)\diamond(y\ast z)+(-1)^{|x||y|}\alpha(y)\diamond(x\ast z).
\end{eqnarray}
\end{defi}
\begin{thm}
Let $(A,\diamond,\ast,\alpha)$ be a Hom-pre-Gerstenhaber algebra. Define
$$x\cdot y=x\diamond y+(-1)^{|x||y|}y\diamond x,\quad [x,y]=x\ast y-(-1)^{(|x|-1)(|y|-1)}y\ast x, \quad \forall x,y \in A.$$
Then $(A,\cdot,[\cdot,\cdot],\alpha)$ is a Hom-Gerstenhaber algebra.
\end{thm}
\begin{proof}
For all $x,y\in A$, we have
\begin{equation}
\nonumber x\cdot y-(-1)^{|x||y|}y\cdot x=x\diamond y+(-1)^{|x||y|}y\diamond x-(-1)^{|x||y|}y\diamond x-(-1)^{2|x||y|}x\diamond y=0,
\end{equation}
which implies that $ x\cdot y=(-1)^{|x||y|}y\cdot x$.

For all $x,y,z\in A$, by \eqref{Hom-zinbiel}, we have
\begin{eqnarray*}&&\alpha(x)\cdot(y\cdot z)-(x\cdot y)\cdot\alpha(z)\\
&=&\alpha(x)\cdot(y\diamond z+(-1)^{|y||z|}z\diamond y)-(x\diamond y+(-1)^{|x||y|}y\diamond x)\cdot\alpha(z)\\
&=&\alpha(x)\diamond(y\diamond z)+(-1)^{|y||z|}\alpha(x)\diamond(z\diamond y)+(-1)^{|x|(|y|+|z|)}(y\diamond z)\diamond\alpha(x)+(-1)^{|x||y|+|x||z|+|y||z|}(z\diamond y)\diamond\alpha(x)\\
&&-(x\diamond y)\diamond\alpha(z)-(-1)^{|x||y|}(y\diamond x)\diamond\alpha(z)-(-1)^{(|x|+|y|)|z|}\alpha(z)\diamond(x\diamond y)-(-1)^{|x||y|+|x||z|+|y||z|}\alpha(z)\diamond(y\diamond x)\\
&=&0,
\end{eqnarray*}
which implies that $\alpha(x)\cdot(y\cdot z)=(x\cdot y)\cdot\alpha(z)$.

For all $x,y\in A$, we have
\begin{equation}
\nonumber \alpha(x \cdot y)=\alpha(x\diamond y+(-1)^{|x||y|}y\diamond x)=\alpha(x)\diamond \alpha(y)+(-1)^{|x||y|}\alpha(y)\diamond \alpha(x)=\alpha(x)\cdot \alpha(y).
\end{equation}
Thus, $(A,\cdot,\alpha)$ is a graded commutative Hom-associative algebra of degree $0$. Similarly, $(A,[\cdot,\cdot],\alpha)$ is a graded Hom-Lie algebra of degree $-1$.

For all $x,y,z\in A$, by \eqref{eq:pre-gerstenhaber-1} and \eqref{eq:pre-gerstenhaber-2}, we have
\begin{eqnarray*}&&[\alpha(x),y\cdot z]-[x,y]\cdot\alpha(z)-(-1)^{(|x|-1)|y|}\alpha(y)\cdot[x,z]\\
&=&[\alpha(x),y\diamond z+(-1)^{|y||z|}z\diamond y]-(x\ast y-(-1)^{(|x|-1)(|y|-1)}y\ast x)\cdot\alpha(z)\\
&&-(-1)^{(|x|-1)|y|}\alpha(y)\cdot(x\ast z-(-1)^{(|x|-1)(|z|-1)}z\ast x)\\
&=&\alpha(x)\ast(y\diamond z)+(-1)^{|y||z|}\alpha(x)\ast(z\diamond y)-(-1)^{(|x|-1)(|y|+|z|-1)}(y\diamond z)\ast\alpha(x)\\
&&-(-1)^{(|x|-1)(|y|+|z|-1)+|y||z|}(z\diamond y)\ast\alpha(x)-(x\ast y)\diamond\alpha(z)+(-1)^{(|x|-1)(|y|-1)}(y\ast x)\diamond\alpha(z)\\
&&-(-1)^{(|x|+|y|-1)|z|}\alpha(z)\diamond(x\ast y)+(-1)^{(|x|+|y|-1)|z|+(|x|-1)(|y|-1)}\alpha(z)\diamond(y\ast x)-(-1)^{(|x|-1)|y|}\alpha(y)\diamond(x\ast z)\\
&&+(-1)^{(|x|-1)(|y|+|z|-1)}\alpha(y)\diamond(z\ast x)-(-1)^{|y||z|}(x\ast z)\diamond\alpha(y)+(-1)^{(|x|-1)(|z|-1)+|y||z|}(z\ast x)\diamond\alpha(y)\\
&=&0,
\end{eqnarray*}
which implies that
$$[\alpha(x),y\cdot z]=[x,y]\cdot\alpha(z)+(-1)^{(|x|-1)|y|}\alpha(y)\cdot[x,z].$$
Therefore, $(A,\cdot,[\cdot,\cdot],\alpha)$ is a Hom-Gerstenhaber algebra.
\end{proof}

\section{Hom-dendriform formal deformations of Hom-zinbiel algebras}\label{sec:def}
In this section, we introduce the notion of Hom-dendriform formal deformations of Hom-zinbiel algebras and show that Hom-pre-Poisson algebras are the corresponding semi-classical limits.
\begin{defi}
Let $(A,\diamond,\alpha)$ be a Hom-zinbiel algebra. A {\bf Hom-dendriform formal deformation} of $A$ is a sequence of bilinear maps $\prec_i, \succ_i:A\otimes A\longrightarrow A$ with $x \prec_0 y=y \succ_0 x=x\diamond y$, such that $\prec_t,\succ_t$ on $A[[t]]$ defined  by
$$x\prec_t y=\sum_{i=0}^{+\infty}x\prec_i y t^i,\quad x\succ_t y=\sum_{i=0}^{+\infty}x\succ_i y t^i, $$
determine a  Hom-dendriform algebra.
\end{defi}

Let $(A,\diamond,\alpha)$ be a Hom-zinbiel algebra and $(A[[t]],\prec_t,\succ_t,\alpha)$ be a Hom-dendriform formal deformation of $A$. For all $x,y,z\in A$ and $n=1,2,\dots$,  we have
\begin{eqnarray}
\label{cocycle-1}\sum_{i+j=n\atop i,j\geq 0}((x\prec_i y)\prec_j\alpha(z)-\alpha(x)\prec_i(y\prec_j z+y\succ_j z))&=&0,\\
\label{cocycle-2}\sum_{i+j=n\atop i,j\geq 0}((x\succ_i y)\prec_j \alpha(z)-\alpha(x)\succ_i(y\prec_j z))&=&0,\\
\label{cocycle-3}\sum_{i+j=n\atop i,j\geq 0}(\alpha(x)\succ_i(y\succ_j z)-(x\prec_i y+x\succ_i y)\succ_j\alpha(z))&=&0.
\end{eqnarray}

\begin{thm}\label{semi-classical-limit}
Let $(A,\diamond,\alpha)$ be a Hom-zinbiel algebra and $(A[[t]],\prec_t,\succ_t,\alpha)$ be  a Hom-dendriform formal deformation of $A$. Define
$$x\ast y=x\succ_1 y-y\prec_1 x, \quad \forall x,y \in A.$$
Then $(A,\diamond,\ast,\alpha)$ is a Hom-pre-Poisson algebra.
\end{thm}
\begin{proof}
For all $x,y\in A$, define $x\ast_t y=x\succ_t y-y\prec_t x$, we obtain that $(A[[t]],\ast_t,\alpha)$ is a Hom-pre-Lie algebra. Thus, for all $x,y,z\in A$, since $(A[[t]],\ast_t,\alpha)$ is a Hom-pre-Lie algebra, we have
\begin{eqnarray}
\nonumber &&\sum_{i+j=n\atop i,j\geq 0}\Big((x\succ_i y-y\prec_i x)\succ_j \alpha(z)-\alpha(z)\prec_j (x\succ_i y-y\prec_i x)\\
\nonumber&&-\alpha(x)\succ_j(y\succ_i z-z\prec_i y)+(y\succ_i z-z\prec_i y)\prec_j \alpha(x)\\
\nonumber&&-(y\succ_i x-x\prec_i y)\succ_j \alpha(z)+\alpha(z)\prec_j (y\succ_i x-x\prec_i y)\\
\nonumber&&+\alpha(y)\succ_j(x\succ_i z-z\prec_i x)-(x\succ_i z-z\prec_i x)\prec_j \alpha(y)\Big)\\
\nonumber&=&0.
\end{eqnarray}
When $n=2$, we have
\begin{eqnarray}
\nonumber &&(x\succ_1 y-y\prec_1 x)\succ_1 \alpha(z)-\alpha(z)\prec_1 (x\succ_1 y-y\prec_1 x)\\
\nonumber&&-\alpha(x)\succ_1(y\succ_1 z-z\prec_1 y)+(y\succ_1 z-z\prec_1 y)\prec_1 \alpha(x)\\
\nonumber&&-(y\succ_1 x-x\prec_1 y)\succ_1 \alpha(z)+\alpha(z)\prec_1 (y\succ_1 x-x\prec_1 y)\\
\nonumber&&+\alpha(y)\succ_1(x\succ_1 z-z\prec_1 x)-(x\succ_1 z-z\prec_1 x)\prec_1 \alpha(y)\\
\nonumber&=&0,
\end{eqnarray}
which implies that
$$(x\ast y)\ast \alpha(z)-\alpha(x)\ast(y\ast z)-(y\ast x)\ast \alpha(z)+\alpha(y)\ast(x\ast z)=0.$$
Thus, we obtain that $(A,\ast,\alpha)$ is a Hom-pre-Lie algebra.

By \eqref{cocycle-1}, \eqref{cocycle-2} and \eqref{cocycle-3}, when $n=1$, we have
\begin{eqnarray}
\label{cocycle-4}&&(x\prec_0 y)\prec_1\alpha(z)-\alpha(x)\prec_0(y\prec_1 z+y\succ_1 z)\\
\nonumber&&+(x\prec_1 y)\prec_0\alpha(z)-\alpha(x)\prec_1(y\prec_0 z+y\succ_0 z)=0,\\
\label{cocycle-6}&&\alpha(x)\succ_1(y\succ_0 z)-(x\prec_1 y+x\succ_1 y)\succ_0\alpha(z)\\
\nonumber&&+\alpha(x)\succ_0(y\succ_1 z)-(x\prec_0 y+x\succ_0 y)\succ_1\alpha(z)=0,\\
\label{cocycle-5}&&(x\succ_1 y)\prec_0 \alpha(z)-\alpha(x)\succ_1(y\prec_0 z)+(x\succ_0 y)\prec_1 \alpha(z)-\alpha(x)\succ_0(y\prec_1 z)=0.
\end{eqnarray}
For all $x,y,z\in A$, by \eqref{cocycle-4}, \eqref{cocycle-6} and \eqref{cocycle-5}, we have
\begin{eqnarray}
\nonumber &&(x\ast y-y\ast x)\diamond\alpha(z)-\alpha(x)\ast(y\diamond z)+\alpha(y)\diamond(x\ast z)\\
\nonumber&=&(x\succ_1 y)\succ_0 \alpha(z)-(y\prec_1 x)\succ_0 \alpha(z)-(y\succ_1 x)\succ_0 \alpha(z)+(x\prec_1 y)\succ_0\alpha(z)\\
\nonumber &&-\alpha(x)\succ_1(y\succ_0 z)+(y\succ_0 z)\prec_1 \alpha(x)+\alpha(y)\succ_0(x\succ_1 z)-\alpha(y)\succ_0(z\prec_1 x)\\
\nonumber&=&(x\succ_1 y)\succ_0 \alpha(z)-(y\prec_1 x)\succ_0 \alpha(z)-(y\succ_1 x)\succ_0 \alpha(z)+(x\prec_1 y)\succ_0\alpha(z)\\
\nonumber &&-\alpha(x)\succ_1(y\succ_0 z)+\alpha(y)\succ_0(x\succ_1 z)+\alpha(y)\succ_1(z\prec_0 x)-(y\succ_1 z)\prec_0\alpha(x)\\
\nonumber&=&(x\succ_1 y)\succ_0 \alpha(z)+(x\prec_1 y)\succ_0\alpha(z)-\alpha(x)\succ_1(y\succ_0 z)-(y\succ_1 z)\prec_0\alpha(x)\\
\nonumber &&+(y\prec_0 x)\succ_1 \alpha(z)+(y\succ_0 x)\succ_1 \alpha(z)\\
\nonumber&=&\alpha(x)\succ_0(y\succ_1 z)-(x\prec_0 y)\succ_1 \alpha(z)-(x\succ_0 y)\succ_1\alpha(z)-(y\succ_1 z)\prec_0\alpha(x)\\
\nonumber &&+(y\prec_0 x)\succ_1 \alpha(z)+(y\succ_0 x)\succ_1 \alpha(z)\\
\nonumber &=&0,
\end{eqnarray}
which implies that
$$(x\ast y-y\ast x)\diamond\alpha(z)=\alpha(x)\ast(y\diamond z)-\alpha(y)\diamond(x\ast z).$$
Similarly, we have
$$(x\diamond y+y\diamond x)\ast \alpha(z)=\alpha(x)\diamond(y\ast z)+\alpha(y)\diamond (x\ast z).$$
Thus, $(A,\diamond,\ast,\alpha)$ is a Hom-pre-Poisson algebra.
\end{proof}
\begin{defi}
The Hom-pre-Poisson algebra $(A,\diamond,\ast,\alpha)$ given in Theorem \ref{semi-classical-limit} is called the {\bf semi-classical limit} of the Hom-dendriform algebra $(A[[t]],\prec_t,\succ_t,\alpha)$ and the Hom-dendriform algebra $(A[[t]],\prec_t,\succ_t,\alpha)$ is called the {\bf Hom-dendriform deformation quantization} of the Hom-zinbiel algebra $(A,\diamond,\alpha)$.

\end{defi}

\section{Hom-$\huaO$-operators on Hom-Poisson algebras}\label{sec:O}
In this section, we introduce the notion of Hom-$\huaO$-operator on Hom-Poisson algebra. On one hand, a Hom-$\huaO$-operator on a Hom-Poisson algebra gives a Hom-pre-Poisson algebra, and on the other hand, a Hom-pre-Poisson algebra naturally gives rise to  a Hom-$\huaO$-operator on the sub-adjacent Hom-Poisson algebra.
\begin{defi}
Let $(A,\cdot,\alpha)$ be a commutative Hom-associative algebra and $(V,\beta,\mu)$ its representation. A linear map $T:V\longrightarrow A$ is called a {\bf Hom-$\huaO$-operator} on $A$ with respect to $(V,\beta,\mu)$ if for all $u,v\in V$, the following equalities are satisfied
\begin{eqnarray}
  \label{eq:opreator-1}T\circ \beta&=&\alpha\circ T,\\
   \label{eq:operator-2}T(u)\cdot T(v)&=&T(\mu(T(\beta^{-1}(u)))(v)+\mu(T(\beta^{-1}(v)))(u)).
\end{eqnarray}
\end{defi}
\begin{lem}\label{O-Opreator-zinbiel}
Let $(A,\cdot,\alpha)$ be a commutative Hom-associative algebra and $(V,\beta,\mu)$ its representation. Let $T:V\longrightarrow A$ be a Hom-$\huaO$-operator on $A$ with respect to $(V,\beta,\mu)$. Then there exists a Hom-zinbiel algebra structure on $V$ given by
$$u\diamond_T v=\mu(T(\beta^{-1}(u)))v, \quad \forall u,v \in V.$$
\end{lem}
\begin{proof}
By \eqref{hom-comm-rep-1} and \eqref{eq:opreator-1}, for all $u,v\in V$, we have
\begin{equation}
\nonumber \beta(u\diamond_T v)=\beta(\mu(T(\beta^{-1}(u)))v)=\mu(\alpha(T(\beta^{-1}(u))))\beta(v)=\mu(T(u))\beta(v)=\beta(u)\diamond_T\beta(v).
\end{equation}
By \eqref{hom-comm-rep-2} and \eqref{eq:operator-2}, for all $u,v,w\in V$, we have
\begin{eqnarray}
\nonumber &&\beta(u)\diamond_T(v\diamond_T w)-(u\diamond_T v)\diamond_T \beta(w)-(v\diamond_T u)\diamond_T \beta(w)\\
\nonumber&=&\beta(u)\diamond_T \mu(T(\beta^{-1}(v)))(w)-\mu(T(\beta^{-1}(u)))(v)\diamond_T \beta(w)-\mu(T(\beta^{-1}(v)))(u)\diamond_T \beta(w)\\
\nonumber&=&\mu(T(u))\mu(T(\beta^{-1}(v)))(w)-\mu(T(\beta^{-1}(\mu(T(\beta^{-1}(u)))v)))\beta(w)-\mu(T(\beta^{-1}(\mu(T(\beta^{-1}(v)))u)))\beta(w)\\
\nonumber&=&\mu(T(u))\mu(T(\beta^{-1}(v)))(w)-\mu(\alpha^{-1}(T(u)\cdot T(v)))\beta(w)\\
\nonumber &=&0,
\end{eqnarray}
which implies that
$$\beta(u)\diamond_T(v\diamond_T w)=(u\diamond_T v)\diamond_T \beta(w)+(v\diamond_T u)\diamond_T \beta(w).$$
This finishes the proof.
\end{proof}

\begin{defi}{\rm(\cite{Cai-Sheng})}\label{O-Opreator-homlie}
Let $(A,[\cdot,\cdot],\alpha)$ be a Hom-Lie algebra and $(V,\beta,\rho)$ its representation. A linear map $T:V\longrightarrow A$ is called a {\bf Hom-$\huaO$-operator} on $A$ with respect to $(V,\beta,\rho)$ if for all $u,v\in V$, the following equalities are satisfied
\begin{eqnarray}
  \label{eq:opreator-3}T\circ \beta&=&\alpha\circ T,\\
   \label{eq:operator-4}[T(u),T(v)]&=&T(\rho(T(\beta^{-1}(u)))(v)-\rho(T(\beta^{-1}(v)))(u)).
\end{eqnarray}
\end{defi}
\begin{lem}\label{O-Opreator-prelie}
Let $(A,[\cdot,\cdot],\alpha)$ be a Hom-Lie algebra and $(V,\beta,\rho)$ its representation. Let $T:V\longrightarrow A$ be a Hom-$\huaO$-operator on $A$ with respect to $(V,\beta,\rho)$. Then there exists a Hom-pre-Lie algebra structure on $V$ given by
$$u\ast_T v=\rho(T(\beta^{-1}(u)))v, \quad \forall u,v \in V.$$
\end{lem}
\begin{proof}
The proof is similar to Lemma \ref{O-Opreator-zinbiel}, we omit the details.
\end{proof}

\begin{defi}
Let $(A,\cdot,[\cdot,\cdot],\alpha)$ be a Hom-Poisson algebra and $(V,\beta,\rho,\mu)$ its representation. A linear map $T:V\longrightarrow A$ is called a {\bf Hom-$\huaO$-operator} on $A$ with respect to $(V,\beta,\rho,\mu)$ if $T$ is both a Hom-$\huaO$-operator on the commutative Hom-associative algebra $(A,\cdot,\alpha)$ and a Hom-$\huaO$-operator on the Hom-Lie algebra $(A,[\cdot,\cdot],\alpha).$
\end{defi}
\begin{pro}\label{O-operator-hom-poisson}
Let  $(A,\diamond,\ast,\alpha)$ be a Hom-pre-Poisson algebra. Then the algebra morphism $\alpha$ is a Hom-$\huaO$-operator on the sub-adjacent Hom-Poisson algebra $A^C$ with respect to the representation $(A,\alpha,\huaL,\mathfrak{L})$, where $\huaL_x(y)=x\ast y$ and $\mathfrak{L}_x(y)=x\diamond y$.
\end{pro}
\begin{proof}
Since $\alpha$ is an algebra morphism of the sub-adjacent Hom-Poisson algebra $A^C$, for all $x,y\in A$, we have
\begin{equation}
\nonumber \alpha(x)\cdot \alpha(y)=\alpha(x\cdot y)=\alpha(x\diamond y+\diamond\diamond x)=\alpha(\mathfrak{L}(\alpha(\alpha^{-1}(x)))y+\mathfrak{L}(\alpha(\alpha^{-1}(y)))x),
\end{equation}
which implies that $\alpha$ is a Hom-$\huaO$-operator on the commutative Hom-associative algebra. Similarly, we can prove that  $\alpha$ is a Hom-$\huaO$-operator on the Hom-Lie algebra. This finishes the proof.
\end{proof}
\begin{thm}\label{O-operator-L}
Let $(A,\cdot,[\cdot,\cdot],\alpha)$ be a Hom-Poisson algebra and $T:V\longrightarrow A$ be a Hom-$\huaO$-operator on $A$ with respect to the representation $(V,\beta,\rho,\mu)$. Define  new multiplications  $``\diamond_T$'' and $``\ast_T$'' by
\begin{equation}
\nonumber u\diamond_T v=\mu(T(\beta^{-1}(u)))v, \quad u\ast_T v=\rho(T(\beta^{-1}(u)))v,\quad\forall u,v \in V.
\end{equation}
Then $(V,\diamond_T,\ast_T,\beta)$ is a Hom-pre-Poisson algebra.
\end{thm}
\begin{proof}
By Lemma \ref{O-Opreator-zinbiel} and Lemma \ref{O-Opreator-prelie}, we know that $(V,\diamond_T,\beta)$ is a Hom-zinbiel algebra and $(V,\ast_T,\beta)$ is a Hom-pre-Lie algebra. By \eqref{rep-2}, \eqref{eq:opreator-3} and \eqref{eq:operator-4}, for all $u,v,w\in V$, we have
\begin{eqnarray}
\nonumber &&(u\ast_T v-v\ast_T u)\diamond_T \beta(w)-\beta(u)\ast_T(v\diamond_T w)+\beta(v)\diamond_T(u\ast_T w)\\
\nonumber&=&(\rho(T(\beta^{-1}(u)))v-\rho(T(\beta^{-1}(v)))u)\diamond_T \beta(w)-\beta(u)\ast_T \mu(T(\beta^{-1}(v)))(w)+\beta(v)\diamond_T \rho(T(\beta^{-1}(u)))(w)\\
\nonumber&=&\mu(T(\beta^{-1}(\rho(T(\beta^{-1}(u)))v)))\beta(w)-\mu(T(\beta^{-1}(\rho(T(\beta^{-1}(v)))u)))\beta(w)\\
\nonumber&&-\rho(T(u))\mu(T(\beta^{-1}(v)))(w)+\mu(T(v))\rho(T(\beta^{-1}(u)))(w)\\
\nonumber&=&\mu(\alpha^{-1}[T(u),T(v)])\beta(w)-\rho(T(u))\mu(T(\beta^{-1}(v)))(w)+\mu(T(v))\rho(T(\beta^{-1}(u)))(w)\\
\nonumber &=&0,
\end{eqnarray}
which implies that
\begin{equation}\label{Hom-poisson-1}
(u\ast_T v-v\ast_T u)\diamond_T \beta(w)=\beta(u)\ast_T(v\diamond_T w)-\beta(v)\diamond_T(u\ast_T w).
\end{equation}
Similarly, we have
\begin{equation}\label{Hom-poisson-2}
(u\diamond_T v+v\diamond_T u)\ast_T \beta(w)=\beta(u)\diamond_T(v\ast_T w)+\beta(v)\diamond_T(u\ast_T w).
\end{equation}
Thus, by \eqref{Hom-poisson-1} and \eqref{Hom-poisson-2}, we obtain that $(V,\diamond_T,\ast_T,\beta)$ is a Hom-pre-Poisson algebra.
\end{proof}

\begin{cor}\label{O-operator-L1}
Let $(A,\cdot,[\cdot,\cdot],\alpha)$ be a Hom-Poisson algebra and $T:V\longrightarrow A$ be a Hom-$\huaO$-operator on $A$ with respect to the representation $(V,\beta,\rho,\mu)$. Then $T(V)=\{T(u)|u\in V\}\subset A$ is a subalgebra of $A$ and there is an induced Hom-pre-Poisson algebra structure on $T(V)$ given by
\begin{equation}
\nonumber  T(u)\diamond T(v)=T(u\diamond_T v), \quad T(u)\ast T(v)=T(u\ast_T v),  \quad \forall u,v \in V.
\end{equation}
\end{cor}

\begin{pro}\label{Operator-L}
Let $(A,\cdot,[\cdot,\cdot],\alpha)$ be a Hom-Poisson algebra. Then there is a compatible Hom-pre-Poisson algebra structure on $A$ if and only if there exists an invertible Hom-$\huaO$-operator on $(A,\cdot,[\cdot,\cdot],\alpha)$.
\end{pro}

\begin{proof}
Let $T$ be an invertible Hom-$\huaO$-operator on $(A,\cdot,[\cdot,\cdot],\alpha)$. By Theorem \ref{O-operator-L} and Corollary \ref{O-operator-L1}, there exists a compatible Hom-pre-Poisson algebra on $T(V)$ given by
\begin{equation}
\nonumber x\diamond y=T(\mu(\alpha^{-1}(x))T^{-1}(y)), \quad x\ast y=T(\rho(\alpha^{-1}(x))T^{-1}(y)),\quad\forall x,y \in A.
\end{equation}

Conversely, let $(A,\diamond,\ast,\alpha)$ be a compatible Hom-pre-Poisson algebra of $(A,\cdot,[\cdot,\cdot],\alpha)$. Then by Proposition \ref{O-operator-hom-poisson}, we know that $\alpha:A\longrightarrow A$ is a Hom-$\huaO$-operator on $(A,\cdot,[\cdot,\cdot],\alpha)$ with respect to the representation $(A,\alpha,\huaL,\mathfrak{L})$.
\end{proof}

Let $A$ be a vector space. For all $\omega \in \wedge^2 A^\ast$, the linear map $\omega^\sharp:A \longrightarrow A^\ast$ is given by
\begin{equation}\label{hessian}
\langle \omega^\sharp(x),y\rangle=\omega(x,y),\quad \forall x,y\in A.
\end{equation}
\begin{thm}
Let $(A,\cdot,[\cdot,\cdot],\alpha)$ be a Hom-Poisson algebra and $\omega$ be a $2$-cocycle of the commutative Hom-associative algebra $(A,\cdot,\alpha)$, i.e.
\begin{equation}\label{comm-2-cocycle}
\omega(x\cdot y,\alpha(z))+\omega(y\cdot z,\alpha(x))+\omega(z\cdot x,\alpha(y))=0,  \quad \forall x,y,z\in A,
\end{equation}
as well as a $2$-cocycle of the Hom-Lie algebra $(A,[\cdot,\cdot],\alpha)$,  i.e.
\begin{equation}\label{lie-2-cocycle}
\omega([x, y],\alpha(z))+\omega([y,z],\alpha(x))+\omega([z,x],\alpha(y))=0,  \quad \forall x,y,z\in A.
\end{equation}
If $\omega$ satisfies
\begin{equation}\label{symplectic-structure}
\omega(x, y)=\omega(\alpha(x), \alpha(y)),  \quad \forall x,y,z\in A,
\end{equation}
then there is a compatible Hom-pre-Poisson algebra structure on $A$ given by
\begin{equation}
\omega(x\diamond y,z)=\omega(y,\alpha^{-1}(x)\cdot\alpha^{-2}(z)),\quad\omega(x \ast y,z)=-\omega(y,[\alpha^{-1}(x),\alpha^{-2}(z)]),  \quad \forall x,y,z\in A.
\end{equation}

\end{thm}
\begin{proof}
By \eqref{hessian} and \eqref{symplectic-structure}, we obtain that $(\omega^\sharp)^{-1}\circ (\alpha^{-1})^\ast=\alpha\circ (\omega^\sharp)^{-1}.$ Thus, we have
\begin{equation}\label{hessian-operator1}
(\omega^\sharp)^{-1}\circ (\alpha^{-1})^\ast\circ (\alpha^{-1})^\ast=\alpha\circ (\omega^\sharp)^{-1}\circ (\alpha^{-1})^\ast.
\end{equation}
For all $x,y,z\in A, \xi,\eta,\gamma\in A^*$, set $x=\alpha((\omega^\sharp)^{-1}(\xi)), y=\alpha((\omega^\sharp)^{-1}(\eta)), z=\alpha((\omega^\sharp)^{-1}(\gamma))$, by \eqref{comm-2-cocycle}, we have
\begin{eqnarray*}&&\langle (\omega^\sharp)^{-1}((\alpha^{-1})^*(\xi))\cdot (\omega^\sharp)^{-1}((\alpha^{-1})^*(\eta))-(\omega^\sharp)^{-1}(\alpha^{-1})^*(-L^\star((\omega^\sharp)^{-1}(\alpha^{-1})^*(\alpha^*(\xi)))\eta\\
&&-L^\star((\omega^\sharp)^{-1}(\alpha^{-1})^*(\alpha^*(\eta)))\xi),(\alpha^{-2})^*(\gamma)\rangle\\
&=&\langle x\cdot y+(\omega^\sharp)^{-1}(\alpha^{-1})^*(L^\star_{(\omega^\sharp)^{-1}(\xi)}\eta+L^\star_{(\omega^\sharp)^{-1}(\eta)}\xi),(\alpha^{-2})^*(\gamma)\rangle \\
&=&\langle x \cdot y,\omega^\sharp(\alpha(z))\rangle-\langle L^\star_{\alpha^{-1}(x)}\omega^\sharp(\alpha^{-1}(y))+L^\star_{\alpha^{-1}(y)}\omega^\sharp(\alpha^{-1}(x)),z \rangle \\
&=&\omega(\alpha(z),x\cdot y)+\langle \omega^\sharp(\alpha^{-1}(y)),\alpha^{-2}(x)\cdot\alpha^{-2}(z)\rangle+\langle \omega^\sharp(\alpha^{-1}(x)),\alpha^{-2}(y)\cdot \alpha^{-2}(z)\rangle\\
&=&\omega(\alpha(z),x\cdot y)+\omega(\alpha(y),x\cdot z)+\omega(\alpha(x),y\cdot z)\\
&=&0,
\end{eqnarray*}
which implies that
\begin{eqnarray}\label{Hessian-o-operator-2}
\nonumber &&(\omega^\sharp)^{-1}((\alpha^{-1})^*(\xi))\cdot (\omega^\sharp)^{-1}((\alpha^{-1})^*(\eta))\\
 &=&(\omega^\sharp)^{-1}(\alpha^{-1})^*(-L^\star((\omega^\sharp)^{-1}(\alpha^{-1})^*(\alpha^*(\xi)))\eta-L^\star((\omega^\sharp)^{-1}(\alpha^{-1})^*(\alpha^*(\eta)))\xi).
\end{eqnarray}
 By \eqref{hessian-operator1} and \eqref{Hessian-o-operator-2}, we deduce that $(\omega^\sharp)^{-1}\circ (\alpha^{-1})^*$ is a Hom-$\huaO$-operator on the commutative Hom-associative algebra $(A,\cdot,\alpha)$ with respect to the representation $(A^*,(\alpha^{-1})^*,-L^\star)$. Similarly, we obtain that  $(\omega^\sharp)^{-1}\circ (\alpha^{-1})^*$ is a Hom-$\huaO$-operator on the Hom-Lie algebra $(A,[\cdot,\cdot],\alpha)$ with respect to the representation $(A^*,(\alpha^{-1})^*,\ad^\star)$. By Proposition \ref{Operator-L}, there is a compatible Hom-pre-Poisson algebra $(A,\diamond,\ast,\alpha)$. Moreover, for all $x,y,z \in A$, we have
\begin{eqnarray*}\omega(x\diamond y,z)
&=&\omega((\omega^\sharp)^{-1}(\alpha^{-1})^*(-L^\star_{\alpha^{-1}(x)}\alpha^*(\omega^\sharp(y))),z)\\
&=&-\langle L^\star_{\alpha^{-1}(x)}\alpha^*(\omega^\sharp(y)),\alpha^{-1}(z)\rangle\\
&=&\langle \omega^\sharp(y),\alpha^{-1}(x)\cdot\alpha^{-2}(z)\rangle\\
&=&\omega(y,\alpha^{-1}(x)\cdot\alpha^{-2}(z)).
\end{eqnarray*}
Similarly, we have $\omega(x\ast y,z)=-\omega(y,[\alpha^{-1}(x),\alpha^{-1}(z)])$. The proof is finished.
\end{proof}

\section{Dual-Hom-pre-Poisson algebras}\label{sec:dual}
In this section, we introduce the notions of  dual-Hom-pre-Poisson algebra and  Hom-average-operator on a Hom-Poisson algebra. We show that a Hom-average-operator on a Hom-Poisson algebra gives rise to a dual-Hom-pre-Poisson algebra.

First, we define Hom-permutative algebras.
\begin{defi}
A {\bf Hom-permutative algebra} is a triple $(A,\bullet,\alpha)$ consisting of a vector space $A$, a linear map $\bullet:A\otimes A\longrightarrow A$ and an algebra morphism $\alpha:A\longrightarrow A$ satisfying:
\begin{equation}
\alpha(x)\bullet (y\bullet z)=(x \bullet y)\bullet \alpha(z)=(y \bullet x)\bullet \alpha(z),\quad \forall x,y \in A.
\end{equation}
\end{defi}
The following Hom-average-operator provide a relationship  between Hom-permutative algebras and commutative Hom-associative algebras.
\begin{defi}
Let $(A,\cdot,\alpha)$ be a regular commutative Hom-associative algebra. A linear map $S:A\longrightarrow A$ is called a {\bf Hom-average-operator} on $A$ if for all $x,y\in A$, the following identities are satisfied
\begin{eqnarray}
  \label{eq:average-opreator-1}S\circ \alpha&=&\alpha\circ S,\\
   \label{eq:average-operator-2}S(x)\cdot S(y)&=&S(S(\alpha^{-1}(x))\cdot y).
\end{eqnarray}
\end{defi}

\begin{lem}\label{average-Opreator-permutative}
Let $(A,\cdot,\alpha)$ be a regular commutative Hom-associative algebra and $S:A\longrightarrow A$ a Hom-average-operator on $A$. Define a multiplication $``\bullet$'' on $A$ by
\begin{equation}
\nonumber x\bullet y=S(\alpha^{-1}(x))\cdot y,\quad \forall x,y \in A.
\end{equation}
Then $(A,\bullet,\alpha)$ is a Hom-permutative algebra.
\end{lem}
\begin{proof}
 For all $x,y\in A$, by \eqref{eq:average-opreator-1}, we have
\begin{equation}
\nonumber \alpha(x\bullet y)=\alpha(S(\alpha^{-1}(x))\cdot y)=S(x)\cdot \alpha(y)=\alpha(x)\bullet \alpha(y).
\end{equation}
For all $x,y,z\in A$, by \eqref{eq:average-opreator-1}, \eqref{eq:average-operator-2} and the definition of commutative Hom-associative algebra, we have
\begin{eqnarray*}
\alpha(x)\bullet(y\bullet z)-(x\bullet y)\bullet\alpha(z)
&=&\alpha(x)\bullet(S(\alpha^{-1}(y))\cdot z)-(S(\alpha^{-1}(x))\cdot y)\bullet\alpha(z)\\
&=&S(x) \cdot (\alpha^{-1}(S(y))\cdot z)-S(\alpha^{-1}(S(\alpha^{-1}(x))\cdot y))\cdot\alpha(z)\\
&=&S(x)\cdot (\alpha^{-1}(S(y))\cdot z)-(\alpha^{-1}(S(x))\cdot\alpha^{-1}(S(y)))\cdot\alpha(z)\\
&=&0,
\end{eqnarray*}
which implies that
$$\alpha(x)\bullet(y\bullet z)=(x\bullet y)\bullet\alpha(z).$$
Similarly, we have
$$(x\bullet y)\bullet\alpha(z)=(y\bullet x)\bullet\alpha(z).$$
Therefore, $(A,\bullet,\alpha)$ is a Hom-permutative algebra.
\end{proof}
\begin{defi}{\rm(\cite{MS2})}
A {\bf Hom-Leibniz algebra} is a triple $(A,\{\cdot,\cdot\},\alpha)$ consisting of a vector space $A$, a linear map $\{\cdot,\cdot\}:A\otimes A\longrightarrow A$ and an algebra morphism $\alpha:A\longrightarrow A$ satisfying:
\begin{equation}
\{\{x,y\},\alpha(z)\}=\{\alpha(x),\{y,z\}\}-\{\alpha(y),\{x,z\}\},\quad \forall x,y \in A.
\end{equation}
\end{defi}
\begin{defi}
Let $(A,[\cdot,\cdot],\alpha)$ be a regular Hom-Lie algebra. A linear map $S:A\longrightarrow A$ is called a {\bf Hom-average-operator} on $A$ if for all $x,y\in A$, the following identities are satisfied
\begin{eqnarray}
  \label{eq:average-opreator-3}S\circ \alpha&=&\alpha\circ S,\\
   \label{eq:average-operator-4}[S(x),S(y)]&=&S[S(\alpha^{-1}(x)),y].
\end{eqnarray}
\end{defi}
\begin{lem}\label{average-Opreator-Leibniz}
Let $(A,[\cdot,\cdot],\alpha)$ be a regular Hom-Lie algebra and $S:A\longrightarrow A$ be a Hom-average-operator on $A$. Define a bracket $\{\cdot,\cdot\}$ on $A$ by
\begin{equation}
\nonumber \{x,y\}=[S(\alpha^{-1}(x)),y],\quad \forall x,y \in A.
\end{equation}
Then $(A,\{\cdot,\cdot\},\alpha)$ is a Hom-Leibniz algebra.
\end{lem}
\begin{proof}
The proof is similar to Lemma \ref{average-Opreator-permutative}, we omit the details.
\end{proof}
\begin{defi}
A {\bf dual-Hom-pre-Poisson algebra} is a quadruple $(A,\bullet,\{\cdot,\cdot\},\alpha)$, where $(A,\bullet,\alpha)$ is a Hom-permutative algebra and $(A,\{\cdot,\cdot\},\alpha)$ is a Hom-Leibniz algebra, such that for all $x,y,z\in A$, the following conditions hold:
\begin{eqnarray}
\label{dual-hom-pre-poisson-1}\{\alpha(x),y\bullet z\}=\{x,y\}\bullet \alpha(z)+\alpha(y)\bullet \{x,z\},\\
\label{dual-hom-pre-poisson-2}\{x\bullet y,\alpha(z)\}=\alpha(x)\bullet\{y,z\}+\alpha(y)\bullet\{x,z\},\\
\label{dual-hom-pre-poisson-3}\{x,y\}\bullet\alpha(z)+\{y,x\}\bullet\alpha(z)=0.
\end{eqnarray}
\end{defi}

\begin{defi}
Let $(A,\cdot,[\cdot,\cdot],\alpha)$ be a regular Hom-Poisson algebra. A linear map $S:A\longrightarrow A$ is called a {\bf Hom-average-operator} on $A$ if $S$ is both a Hom-average-operator on the commutative Hom-associative algebra $(A,\cdot,\alpha)$ and a Hom-average-operator on the Hom-Lie algebra $(A,[\cdot,\cdot],\alpha).$
\end{defi}

\begin{thm}
Let $(A,\cdot,[\cdot,\cdot],\alpha)$ be a regular Hom-Poisson algebra and $S:A\longrightarrow A$ a Hom-average-operator on $A$. Define $``\bullet$'' and $``\{\cdot,\cdot\}$'' by
\begin{equation}
\nonumber x\bullet y=S(\alpha^{-1}(x))\cdot y, \quad \{x,y\}=[S(\alpha^{-1}(x)),y],\quad\forall x,y \in A.
\end{equation}
Then $(A,\bullet,\{\cdot,\cdot\},\alpha)$ is a dual-Hom-pre-Poisson algebra.
\end{thm}
\begin{proof}
By Lemma \ref{average-Opreator-permutative} and Lemma \ref{average-Opreator-Leibniz}, we know that $(A,\bullet,\alpha)$ is a Hom-permutative algebra and $(A,\{\cdot,\cdot\},\alpha)$ is a Hom-Leibniz algebra. By \eqref{hom-poisson}, \eqref{eq:average-opreator-1}, \eqref{eq:average-operator-2} and \eqref{eq:average-operator-4}, for all $x,y,z\in A$, we have
\begin{eqnarray}
\nonumber &&\{\alpha(x),y\bullet z\}-\{x,y\}\bullet \alpha(z)-\alpha(y)\bullet \{x,z\}\\
\nonumber&=&\{\alpha(x),S(\alpha^{-1}(y))\cdot z\}-[S(\alpha^{-1}(x)),y]\bullet \alpha(z)-\alpha(y)\bullet[S(\alpha^{-1}(x)),z]\\
\nonumber&=&[S(x),S(\alpha^{-1}(y))\cdot z]-S(\alpha^{-1}[S(\alpha^{-1}(x)),y])\cdot \alpha(z)-S(y)\cdot[S(\alpha^{-1}(x)),z]\\
\nonumber&=&[S(x),\alpha^{-1}(S(y))\cdot z]-[\alpha^{-1}(S(x)),\alpha^{-1}(S(y))]\cdot \alpha(z)-S(y)\cdot[\alpha^{-1}(S(x)),z]\\
\nonumber &=&0,
\end{eqnarray}
which implies that
\begin{equation}\label{dual-1}
\{\alpha(x),y\bullet z\}=\{x,y\}\bullet \alpha(z)+\alpha(y)\bullet \{x,z\}.
\end{equation}
Similarly, we have
\begin{eqnarray}
\label{dual-2}\{x\bullet y,\alpha(z)\}=\alpha(x)\bullet\{y,z\}+\alpha(y)\bullet\{x,z\},\\
\label{dual-3}\{x,y\}\bullet\alpha(z)+\{y,x\}\bullet\alpha(z)=0.
\end{eqnarray}
By \eqref{dual-1}, \eqref{dual-2} and \eqref{dual-3}, we obtain that $(A,\bullet,\{\cdot,\cdot\},\alpha)$ is a dual-Hom-pre-Poisson algebra.
\end{proof}


\begin{thebibliography}{999}
\bibitem{AGUIAR} M. Aguiar, Pre-Poisson algebras.  \emph{Lett. Math. Phys.} 54 (2000), no. 4, 263-277.

\bibitem{AEM} F. Ammar, Z. Ejbehi and A. Makhlouf, Cohomology and deformations
of Hom-algebras. \emph{J. Lie Theory.} 21 (2011), no. 4, 813-836.

 \bibitem{Pre-lie algebra in geometry} D. Burde, Left-symmetric algebras and pre-Lie algebras in
geometry and physics, {\em Cent. Eur. J. Math.} 4 (2006), 323-357.

 \bibitem{Bai}C. Bai, Left-symmetric Bialgebras and Analogue of the Classical Yang-Baxter Equation. \emph{Commun. Contemp. Math.} 10 (2008), no. 2, 221-260.


\bibitem{BM}
S. Benayadi and A. Makhlouf, Hom-Lie algebras with symmetric invariant nondegenerate bilinear forms. \emph{J. Geom. Phys.} 76 (2014), 38-60.


\bibitem{BCMM} A. Ben Hassin, T. Chtioui, M. A. Maalaoui and S.  Mabrouk, On Hom-$F$-manifold algebras and quantization. 	arXiv:2102.05595.

\bibitem{Cai-Sheng}
L. Cai and Y. Sheng, Purely Hom-Lie bialgebras.
\emph{Sci. China Math.} 61 (2018), no. 9, 1553-1566.






\bibitem{HLS}
J. Hartwig, D. Larsson and S. Silvestrov, Deformations of Lie
algebras using $\sigma$-derivations. \emph{J. Algebra} 295 (2006),
314-361.





\bibitem{LGT}
C. Laurent-Gengoux and J. Teles, Hom-Lie algebroids. \emph{J. Geom. Phys.} 68 (2013), 69-75.

\bibitem{Loady}
J.-L. Loday, Cup product for Leibniz cohomology and dual Leibniz algebras. Math. Scand. 77, Univ. Louis
Pasteur, Strasbourg, 1995, pp. 189-196.






\bibitem{LJF}
J. Liu, Y. Sheng and C. Bai, $F$-manifold algebras and deformation quantization via pre-Lie algebras. \emph{J. Algebra} 559 (2020), 467--495.

\bibitem{LSS}
S. Liu, A. Makhlouf and L. Song, On Hom-pre-Lie bialgebras. \emph{J. Lie Theory} 31 (2021), no. 1, 149-168.

\bibitem{LST}
S. Liu, L. Song  and   R. Tang, Representations and cohomologies of regular Hom-pre-Lie algebras. \emph{J. Algebra Appl.} 19 (2020), no. 8, 2050149, 22 pp.

\bibitem{MakDend} A. Makhlouf,  Hom-dendriform algebras and Rota-Baxter Hom-algebras, In: Bai, C., Guo, L., Loday, J.-L. (eds.), \emph{Nankai Ser. Pure Appl. Math. Theoret. Phys.,} 9, World Sci. Publ. 147--171 (2012)

\bibitem{MS2} A. Makhlouf and S. Silvestrov, Hom-algebra structures. \emph{J. Gen. Lie Theory Appl.}  2 (2) (2008),   51--64.


\bibitem{MS3} A. Makhlouf and S. Silvestrov, Hom-algebras and Hom-coalgebras. \emph{J. Algebra Appl.} 9(4) (2010), 1--37.

\bibitem{MS1}
 A. Makhlouf and S. Silvestrov, Notes on 1-parameter formal deformations of Hom-associative and Hom-Lie algebras. \emph{Forum Math.} 22 (2010), no. 4, 715--739.



\bibitem{sheng3}
Y. Sheng, Representations of Hom-Lie algebras. \emph{Algebr. Represent. Theory} 15 (2012), no. 6, 1081--1098.

\bibitem{sheng1}
Y. Sheng and C. Bai, A new approach to hom-Lie bialgebras. \emph{J. Algebra} 399 (2014), 232--250.

\bibitem{SC}
Y. Sheng and D. Chen, Hom-Lie 2-algebras. \emph{J. Algebra} 376 (2013), 174--195.

\bibitem{sunbing}
B. Sun, L. Chen and X. Zhou, On universal $\alpha$-central extensions of Hom-pre-Lie algebras. arXiv:1810.09848.

\bibitem{QH}
Q. Sun and H. Li, On parak\"{a}hler hom-Lie algebras and hom-left-symmetric bialgebras. \emph{Comm. Algebra} 45 (2017), no. 1, 105--120.




\bibitem{Yao1}
 D. Yau, The Hom-Yang-Baxter equation, Hom-Lie algebras, and quasi-triangular bialgebras. \emph{J. Phys. A.} 42 (2009), no. 16, 165202, 12 pp.


\bibitem{Yao3}
 D. Yau, Non-commutative Hom-Poisson algebras. arXiv:1010.3408v1.

\bibitem{Qing}
Q. Zhang, H. Yu and C. Wang, Hom-Lie algebroids and hom-left-symmetric algebroids. \emph{J. Geom. Phys.} 116 (2017), 187-203.

\end{thebibliography}
 \end{document}